\documentclass[11pt]{amsart}
\usepackage{amsmath}
\usepackage{amssymb}
\usepackage{amscd}

\def\NZQ{\mathbb}               
\def\NN{{\NZQ N}}

\def\ZZ{{\NZQ Z}}
\def\RR{{\NZQ R}}

\newtheorem{Theorem}{Theorem}[section]
\newtheorem{Lemma}[Theorem]{Lemma}
\newtheorem{Corollary}[Theorem]{Corollary}

\newtheorem{Remark}[Theorem]{Remark}

\newtheorem{Example}[Theorem]{Example}

\newtheorem{Definition}[Theorem]{Definition}

\let\epsilon\varepsilon
\let\phi=\varphi
\let\kappa=\varkappa

\begin{document}
\title{Multiplicities Associated to Graded Families of Ideals}
\author{Steven Dale Cutkosky }
\thanks{Partially supported by NSF}

\address{Steven Dale Cutkosky, Department of Mathematics,
University of Missouri, Columbia, MO 65211, USA}
\email{cutkoskys@missouri.edu}

\begin{abstract}
We prove that limits of multiplicities associated to graded families of ideals exist under very general conditions. Most of our results hold for analytically unramified  equicharacteristic local rings, with perfect residue fields. We give a number of applications, including a "volume = multiplicity" formula, generalizing the formula of Lazarsfeld and Mustata, and a proof that the epsilon multiplicity of Ulrich and Validashti exists as a limit for ideals in rather general rings, including analytic local domains. We prove a generalization of this to generalized symbolic powers of ideals, proposed by Herzog, Puthenpurakal and Verma. We also prove an asymptotic "additivity formula" for limits of multiplicities, and a formula on limiting growth of valuations, which answers a question posed by the author, Kia Dalili and Olga Kashcheyeva. Our proofs are inspired by a philosophy of Okounkov, for computing limits of multiplicities as the volume of a slice of an appropriate cone generated by a semigroup determined by an appropriate filtration on a family of algebraic objects. 
\end{abstract}

\maketitle

\section{Introduction}  
In a series of papers, Okounkov interprets the asymptotic multiplicity of graded families of algebraic objects in terms of the volume of a slice of a corresponding cone (the Okounkov body).  Okounkov's method  employs an asymptotic version of a result of Khovanskii for finitely generated semigroups \cite{Kh}. One of his realizations of this philosophy \cite{Ok1}, \cite{Ok}  gives a construction which computes the volume of a family of graded linear systems. This method was systematically developed by Lazarsfeld and Mustata in \cite{LM}, where many interesting consequences are given, including
a new proof of Fujita approximation (the original proof is in \cite{Fuj}), and the fact that the volume of a big divisor on an irreducible  projective variety over an algebraically closed field is 
a limit, which was earlier proven in \cite{La} using Fujita approximation. More recently, Fulger \cite{Fu}, has extended this result to compute local volumes of   divisors on a log resolution of a normal variety over an algebraically closed field.
Kaveh and Khovanskii have recently greatly generalized the theory of Newton-Okounkov bodies, and  applied this to general graded families of linear systems \cite{KK}.

The method used in these papers is to  choose a nonsingular closed point $\beta$ on the $d$-dimensional variety $X$, and then using a flag, a sequence of subvarieties
$$
\{\beta\}=X_0\subset X_1\subset \cdots \subset X_{d-1}\subset X
$$
which are nonsingular at $\beta$,
to determine a rank $d$ valuation of the function field $k(X)$ that dominates the regular local ring $\mathcal O_{X,\beta}$. This valuation gives a very simple
filtration of $\mathcal O_{X,\beta}$, represented  by monomials in a regular system of parameters of $\mathcal O_{X,p}$, which are   local equations of the flag.
Since the residue field is algebraically closed, this allows us to associate a set of points in $\ZZ^d$  to a linear system on $X$
(by means of a $k$-subspace of $k(X)$ giving the linear system), 
so that the number of these points is equal to the dimension of the linear system. In this way,   a semigroup
in $\ZZ^{d+1}$ is associated to a graded family of linear systems.

One of their applications is to prove a formula of equality of volume and multiplicity for a  graded family $\{I_i\}_{i\in\NN}$ of $m_R$-primary ideals in a
local ring $(R, m_R)$ such that $R$ is a local domain which is essentially of finite type over an algebraically closed field $k$ with $R/m_R=k$ (Theorem 3.8 \cite{LM}).
These assumptions on $R$ are all necessary for their proof. The proof involves interpreting the problem in terms of graded families of linear systems on
a projective variety $X$ on which $R$ is the local ring of a closed point $\alpha$. Then a valuation as above is constructed which is centered at a nonsingular point $\beta\in X$, and the cone methods are used to prove the limit. The formula ``volume equals multiplicity'' for graded families of ideals was first proven by
Ein, Lazarsfeld and Smith for valuation ideals associated to an Abhyankar valuation in a regular local ring which is essentially of finite type over a field in \cite{ELS}. Mustata proved the formula for regular local rings containing a field in \cite{Mus}. In all of these cases, the volume ${\rm vol}(I_*)$
of the family, which is defined as a limsup, is shown to be a limit.

Let $\{I_i\}$ be  a graded family of ideals in a $d$-dimensional (Noetherian) local ring $(R,m_R)$; that is, the family is indexed by the natural numbers, with $I_0=R$ and
$I_iI_j\subset I_{i+j}$ for all $i,j$. Suppose that the ideals are $m_R$-primary (for $i>0$). Let $\ell_R(N)$ denote the length of an $R$-module $N$.
We find very general conditions on $R$ under which the ``volume''
$$
{\rm vol}(I_*)= \limsup\frac{\ell_R(R/I_n)}{n^d/d!}
$$
is actually a limit. For instance, we show that this limit exists if $R$ is analytically unramified and equicharacteristic with perfect residue field (Theorem \ref{Theorem2}),
or if $R$ is regular (Theorem  \ref{Theorem0}). 

We thank the referee for pointing out that our basic result Theorem \ref{Theorem1} is valid without our original assumption of excellence.

Our proof involves reducing to the case of a complete domain, and then finding a suitable valuation which dominates $R$
to construct an Okounkov body. The valuation which we use is of rank 1 and rational rank $d$. There are two issues which require special care in the proof.
The first issue is to reduce to the case of an analytically irreducible domain.
Analytic irreducibility is necessary to handle the boundedness restriction on the corresponding cone (condition (\ref{Cone2})). The proof of boundedness is accomplished by using Huebl's linear Zariski subspace theorem \cite{Hu} (which is valid if $R$ is assumed excellent), or as was pointed out by the referee, by an application of Rees' version of Izumi's theorem \cite{R2}, for which excellence is not required. The second  issue is to handle the case of a nonclosed residue field. Our method
for converting the problem into a problem of cones requires that the residue field of the valuation ring be equal to the residue field of $R$. Care needs to be taken when the base field is not algebraically closed. The perfect condition in Theorem \ref{Theorem2} on the residue field is to prevent the introduction of nilpotents upon base change.

In the case when $I_n=I^n$ with $I$ an $m_R$-primary ideal, the limit $\lim_{ n \rightarrow \infty}\frac{\ell_R(R/I^n)}{n^d/d!}$ is just the Hilbert-Samuel
multiplicity $e(I)$, which is a positive integer. 
In general, when working with the kind of generality allowed by a graded family of $m_R$-primary ideals, the limit will be irrational. For instance, given $\lambda\in \RR_+$, the graded family of $m_R$-primary ideals $I_n$ generated by the monomials $x^iy^j$ such that $\frac{1}{2\lambda}i+j\ge n$ in the power series ring $R=k[[x,y]]$ in two variables will give us the
limit $\lim_{n\rightarrow \infty}\frac{\ell_R(R/I_n)}{n^2}=\lambda$.

We also obtain irrational limits for  more classical  families of ideals.
Suppose that $R$ is an excellent $d$-dimensional local domain with perfect residue field, and $\nu$ is a discrete valuation dominating $R$ (the value group is $\ZZ$).
Then the valuation ideals $I_n=\{f\in R\mid \nu(f)\ge n\}$ form a graded family of $m_R$-primary ideals, so Theorem \ref{Theorem2} tells us that the limit
$\lim_{n\rightarrow \infty}\frac{\ell_R(R/I_n)}{n^d}$ exists. This limit will however in general not be rational. Example 6 of  \cite{CS} gives such an example, in a three dimensional normal local ring.

We give a number of applications of this formula and these techniques to the computation of  limits in commutative algebra.

We prove the formula ``${\rm vol}(I_*)={\rm multiplicity}(I_*)$''' for local rings $R$ and graded families of $m_R$-primary ideals such that  either $R$ is regular, or $R$ is analytically unramified
and equicharacteristic with perfect residue field in Theorem \ref{Theorem5}. In our proof, we use a critical result on volumes of cones, which is derived in \cite{LM}. We generalize this result to obtain an asymptotic additivity formula for multiplicities of an arbitrary graded family of ideals (not required to be  $m_R$-primary) in Theorem \ref{Theorem3}.

Another application is to show that the epsilon multiplicity of Ulrich and Validashti \cite{UV}, defined as a limsup, is actually a limit in some new situations. 
We prove that this limit exists for graded families of ideals, in a  local ring $R$ such that one of the following holds: $R$ is regular, $R$ is analytically irreducible and excellent with algebraically closed residue field or $R$ is normal, excellent and equicharacteristic with perfect residue field.
As an immediate consequence, we obtain the existence of the limit for graded families of ideals in an analytic local domain, which is of interest in
singularity theory. 
In \cite{CHST}, an example is given showing that this limit is in general not rational. 
Previously, the limit was shown to exist in some cases in \cite{CHS}, and the existence of the limit was proven (for more general modules) over a domain $R$ which is essentially of finite type over
a perfect field in \cite{C}. The proof used Fujita approximation on a projective variety on which the ring $R$ was the local ring of a closed point.

We proof in Corollary \ref{Cor5} a formula on asymptotic multiplicity of generalized symbolic powers, proposed by Herzog, Puthenpurakal and Verma in the beginning  of
the introduction of \cite{HPV}.

We also prove that a question raised in \cite{CDK} about the growth of the semigroup of a valuation semigroup has a positive answer for very general valuations and domains. We prove in Theorem \ref{Theorem8} that if $R$ is a $d$-dimensional regular local ring or an analytically unramified local domain with algebraically closed residue field, and $\omega$ is a zero dimensional
rank one valuation dominating $R$, with value group contained in $\RR$, and if $\phi(n)$ is the number of elements in the semigroup of values attained on $R$
which are $<n$, then 
$$
\lim_{n\rightarrow \infty}\frac{\phi(n)}{n^d}
$$
exists. This formula was established if $R$ is a regular local ring of dimension 2 with algebraically closed residue field in \cite{CDK}, and if
$R$ is an arbitrary  regular local ring of dimension 2 in \cite{CV} using a detailed analysis of a generating sequence associated to the valuation.
Our proof of this result in general dimension follows, as an application of the existence of limits for graded families of $m_R$-primary ideals,  from the fact that $\phi(n)=\ell_R(R/I_n)$, where $I_n=\{f\in m_R\mid \nu(f)\ge n\}$ (\cite{CDK} and \cite{CT}). 
It is shown in \cite{CDK} that the limits 
$\lim_{n\rightarrow \infty}\frac{\phi(n)}{n^2}$ which can be attained on a regular local ring of dimension 2 are the real numbers $\beta$ with
$0\le \beta<\frac{1}{2}$.

We thank the referee for their careful reading of this paper, and for  suggesting that we  present the theorems
with the less restrictive assumption  of analytically unramified, instead of reduced and excellent.

\section{notation} $m_R$ will denote the maximal ideal of a local ring $R$. $Q(R)$ will denote the quotient field of a domain $R$.
$\ell_R(N)$ will denote the length of an $R$-module $N$.  $\ZZ_+$ denotes the positive integers and $\NN$ the nonnegative integers. 
Suppose that $x\in \RR$. $\lceil x \rceil$ is the smallest integer $n$ such  that $x\le n$. $\lfloor x \rfloor$ is the largest integer $n$ such that $n\le x$. 

We recall some notation on multiplicity from Chapter VIII, Section 10 of \cite{ZS2}, Section V-2 \cite{Se} and Section 4.6 \cite{BH}.
Suppose that $(R,m_R)$ is a (Noetherian) local ring,  $N$ is a finitely generated $R$-module with $r=\dim N$ and $a$ is an ideal of definition of $R$. Then
$$
e_a(N)=\lim_{k\rightarrow \infty}\frac{\ell_R(N/a^kN)}{k^r/r!}.
$$
We write $e(a)=e_a(R)$.

If $s\ge r=\dim N$, then we define 
$$
e_s(a,N)=\left\{
\begin{array}{ll}
 e_a(N)&\mbox{ if }\dim N=s\\
 0&\mbox{ if } \dim N<s.
 \end{array}\right.
 $$

A local ring is analytically unramified if its completion is reduced. In particular, a reduced excellent local ring is 
analytically unramified.

\section{Semigroups and Cones}

Suppose that $\Gamma\subset \NN^{d+1}$ is a semigroup. Set 
$$
\Sigma = \Sigma(\Gamma)=\mbox{ closed convex cone}(\Gamma)\subset \RR^{d+1},
$$
$$
\Delta=\Delta(\Gamma)=\Sigma\cap(\RR^d\times\{1\}).
$$
For $m\in \NN$, put
$$
\Gamma_m=\Gamma\cap (\NN^d\times\{m\}).
$$
which can be viewed as a subset of $\NN^d$. Consider the following three conditions on $\Gamma$:
\begin{equation}\label{Cone1}
\Gamma_0=\{0\}
\end{equation}
\begin{equation}\label{Cone2}
\begin{array}{l}
\mbox{There exist finitely many vectors $(v_i,1)$ spanning a semigroup $B\subset\NN^{d+1}$}\\
\mbox{such that $\Gamma\subset B$}
\end{array}
\end{equation}
Let $G(\Gamma)$ be the subgroup of $\ZZ^{d+1}$ generated by $\Gamma$. 
\begin{equation}\label{Cone3}
G(\Gamma)=\ZZ^{d+1}
\end{equation}

We will use the convention that $\{e_i\}$ is the standard basis of $\ZZ^{d+1}$.

Assuming the boundedness condition (\ref{Cone2}), condition (\ref{Cone1}) simply states that $0$ is in the semigroup $\Gamma$.

\begin{Theorem}\label{ConeTheorem1}(Section 3, \cite{Ok}, Proposition 2.1 \cite{LM}) Suppose that $\Gamma$ satisfies (\ref{Cone1}) - (\ref{Cone3}). Then
$$
\lim_{m\rightarrow\infty} \frac{\# \Gamma_m}{m^d}={\rm vol}(\Delta(\Gamma)).
$$
\end{Theorem}

Recently, it has been shown that limits exist under much weaker conditions by Kaveh and Khovanskii in \cite{KK}.

\begin{Theorem}\label{ConeTheorem2}(Proposition 3.1 \cite{LM}) Suppose that $\Gamma$ satisfies (\ref{Cone1}) - (\ref{Cone3}). Fix $\epsilon>0$. Then there is an integer $p_0= p_0(\epsilon)$
such that if $p\ge p_0$, then the limit
$$
\lim_{k\rightarrow \infty} \frac{\#(k\Gamma_p)}{k^dp^d}\ge {\rm vol}(\Delta(\Gamma))-\epsilon
$$
exists, where
$$
k\Gamma_p=\{x_1+\cdots+x_k\mid x_1,\ldots,x_k\in \Gamma_p\}.
$$
\end{Theorem}

\section{An asymptotic theorem on lengths}

\begin{Definition} A graded family of ideals $\{I_i\}$ in a ring $R$ is a family of ideals indexed by the natural numbers such that $I_0=R$ and $I_iI_j\subset I_{i+j}$
for all $i,j$.  If $R$ is a local ring and $I_i$ is $m_R$-primary for $i>0$, then we will say that $\{I_i\}$ is a graded family of $m_R$-primary ideals.
\end{Definition}

In this section we prove the following theorem.
\begin{Theorem}\label{Theorem1} Suppose that $R$ is an analytically irreducible local domain of dimension $d>0$ and  $\{I_n\}$ is a graded family of  ideals in $R$ such that 
\begin{equation}\label{eq8}
\mbox{there exists $c\in \ZZ_+$ such that $m_R^c\subset I_1$.}
\end{equation} 
Suppose that there exists a regular local ring $S$ such that $S$ is essentially of finite type  and birational over $R$ ($R$ and $S$ have the same  quotient field)   and the residue field map 
$R/m_R\rightarrow S/m_S$ is an isomorphism. Then
$$
\lim_{i\rightarrow\infty}\frac{\ell_R(R/I_i)}{i^d}
$$
exists.
\end{Theorem}

We remark that the assumption $m_R^c\subset I_1$ implies that either $I_n$ is $m_R$-primary for all positive $n$, or there exists
$n_0>1$ such that $I_{n_0}=R$. In this case, $m_R^{cn_0}\subset I_n$ for all $n\ge n_0$, so $\ell_R(R/I_i)$ is actually bounded.

Let assumptions be as in Theorem \ref{Theorem1}. 
Let $y_1,\ldots,y_d$ be a regular system of parameters in $S$. Let $\lambda_1,\ldots,\lambda_d$ be rationally independent real numbers, such that 
\begin{equation}\label{eq9}
\lambda_i\ge 1\mbox{ for all $i$}.
\end{equation}
 We define a valuation $\nu$ on
$Q(R)$ which dominates $S$ by prescribing 
$$
\nu(y_1^{a_1}\cdots y_d^{a_d})=a_1\lambda_1+\cdots+a_d\lambda_d
$$
for $a_1,\ldots,a_d\in \ZZ_+$, and $\nu(\gamma)=0$  if $\gamma\in S$ has nonzero residue in $S/m_S$.

Let $C$ be a coefficient set of $S$. Since $S$ is a regular local ring, for $r\in \ZZ_+$ and $f\in S$, there is a unique expression 
$$
f=\sum s_{i_1,\ldots,i_d}y_1^{i_1}\cdots y_d^{i_d}+g_r
$$
with $g_r\in m_S^r$, $s_{i_1,\ldots,i_d}\in S$ and $i_1+\cdots+i_d<r$ for all $i_1,\ldots,i_d$ appearing in the sum. Take $r$ so large that 
$r> i_1\lambda_1+\cdots+i_d\lambda_d$ for some term with $s_{i_1,\ldots,i_d}\ne 0$. Then define
\begin{equation}\label{eq61}
\nu(f)=\min\{i_1\lambda_1+\cdots+i_d\lambda_d\mid s_{i_1,\ldots,i_d}\ne 0\}.
\end{equation}
This definition is well defined, and we calculate that
$\nu(f+g)\ge \min\{\nu(f),\nu(g)\}$ and $\nu(fg)=\nu(f)+\nu(g)$ (by the uniqueness of the expansion (\ref{eq61})) for all $0\ne f,g\in S$. Thus $\nu$ is a valuation.
 Let $V_{\nu}$ be the valuation ring of $\nu$ (in $Q(R)$). The value group of $V_{\nu}$ is the (nondiscrete) ordered subgroup 
$\ZZ\lambda_1+\cdots+\ZZ\lambda_d$ of $\RR$. Since there is unique monomial giving the minimum in (\ref{eq61}), we have that the residue field of $V_{\nu}$ is
$S/m_S=R/m_R$.

For $\lambda\in \RR$, define  ideals $K_{\lambda}$ and $K_{\lambda^+}$ in $V_{\nu}$ by
$$
K_{\lambda}=\{f\in Q(R)\mid \nu(f)\ge\lambda\}
$$
and
$$
K_{\lambda^+}=\{f\in Q(R)\mid \nu(f)>\lambda\}.
$$

We follow the usual convention that $\nu(0)=\infty$ is larger than any element of $\RR$.

\begin{Lemma}\label{Lemma7} There exists $\alpha\in \ZZ_+$ such that $K_{\alpha n}\cap R\subset m_R^n$ for all $n\in \NN$.
\end{Lemma}

\begin{proof} Let $\rho=\lceil \max\{\lambda_1,\ldots,\lambda_d\}\rceil\in \ZZ_+$. Suppose that $\lambda\in \RR_+$. $K_{\lambda}$ is generated by the monomials
$y_1^{i_1}\cdots y_d^{i_d}$ such that $i_1\lambda_1+\cdots+ i_d\lambda_d\ge \lambda$, which implies that
$$
\frac{\lambda}{\rho}\le i_1+\cdots+i_d,
$$
so that
\begin{equation}\label{eq10}
K_{\lambda}\cap S\subset m_S^{\lceil \frac{\lambda}{\rho}\rceil}.
\end{equation}
We now establish the following equation:
there exists $a\in \ZZ_+$ such that
\begin{equation}\label{eq11}
m_S^{a\ell}\cap R\subset m_R^{\ell}
\end{equation}
for all $\ell\in\NN$.

In the case when $R$ is excellent, this is immediate from the linear Zariski subspace theorem, Theorem 1 of \cite{Hu}.

We now give a proof of (\ref{eq11}) which was provided by the referee, which is valid without assuming that $R$ is excellent. 
Let $\omega$ be the $m_S$-adic valuation. Let $\nu_i$ be the Rees valuations of $m_R$. The $\nu_i$ extend uniquely to the Rees valuations of $m_{\hat R}$. By Rees' version of Izumi's theorem, \cite{R2}, the topologies defined on $R$ by $\omega$ and the $\nu_i$ are linearly equivalent. Let $\overline \nu_{m_R}$ be the reduced order of $m_R$. By the Rees valuation theorem (recalled in \cite{R2}),
$$
\overline \nu_{m_R}(x)=\min_i\left\{\frac{\nu_i(x)}{\nu_i(m_R)}\right\}
$$
for all $x\in R$, so the topology defined by $\omega$ on $R$ is linearly equivalent to the topology defined  by $\overline\nu_{m_R}$.
The $\overline \nu_{m_R}$ topology is linearly equivalent to the $m_R$-topology by \cite{R1}, since $R$ is analytically unramified.
Thus (\ref{eq11}) is established.

Let $\alpha=\rho a$, where $\rho$ is the constant of (\ref{eq10}), and $a$ is the constant of (\ref{eq11}).
$$
K_{\alpha n}\cap S=K_{\rho a n}\cap S\subset m_S^{an}
$$
by (\ref{eq10}), and thus
$$
K_{\alpha n}\cap R\subset m_S^{an}\cap R\subset m_R^n
$$
by (\ref{eq11}).

\end{proof}

\begin{Remark} The conclusions of Lemma \ref{Lemma7} fail if $R$ is not analytically irreducible, as can be seen from the example
$$
R=\left(k[x,y]/y^2-x^2(x+1)\right)_{(x,y)}\rightarrow S=k[s]_{(s)},
$$
where $s=\frac{y}{x}-1$.
\end{Remark}

For $0\ne f\in R$, define 
$$
\phi(f)=(n_1,\ldots,n_d)\in \NN^d
$$
if $\nu(f)=n_1\lambda_1+\cdots+n_d\lambda_d$.

\begin{Lemma} 
Suppose that $I\subset R$ is an ideal and $\lambda\in\RR_+$. Then there are isomorphisms of $R/m_R$-modules 
$$
K_{\lambda}\cap I/K_{\lambda^+}\cap I\cong \left\{ 
\begin{array}{ll}
k&\mbox{ if there exists $f\in I$ such that $\nu(f)=\lambda$}\\
0&\mbox{ otherwise}.
\end{array}\right.
$$
\end{Lemma}

\begin{proof} Suppose that $f,g\in K_{\lambda}\cap I$ are such that $\nu(f)=\nu(g)=\lambda$. Then $\nu(\frac{f}{g})=0$. Let $\overline\alpha$ be the class
of $\frac{f}{g}$ in $V_{\nu}/m_{\nu}\cong R/m_R$. Let $\alpha\in R$ be a lift of $\overline\alpha$ to $R$. Then
$\nu(f-\alpha g)>\lambda$, and the class of $f$ in $K_{\lambda}\cap I/K_{\lambda^+}\cap I$ is equal to $\overline\alpha$ times the class of $g$ in
$K_{\lambda}\cap I/K_{\lambda^+}\cap I$.
\end{proof}

Suppose that $I\subset R$ is an ideal and $K_{\beta}\cap R \subset I$ for some $\beta\in \RR_+$. Then
\begin{equation}\label{eq12}
\begin{array}{lll}
\ell_R(R/I)&=&\ell_R(R/K_{\beta}\cap R)-\ell_R(I/K_{\beta}\cap R)\\
&=& \dim_k\left(\bigoplus_{\lambda<\beta}K_{\lambda}\cap R/K_{\lambda^+}\cap R\right)
-\dim_k\left(\bigoplus_{\lambda<\beta}K_{\lambda}\cap I/K_{\lambda^+}\cap R\right)\\
&=& \#\{(n_1,\ldots,n_d)\in \phi(R)\mid n_1\lambda_1+\cdots+n_d\lambda_d<\beta\}\\
&& - \#\{(n_1,\ldots,n_d)\in \phi(I)\mid n_1\lambda_1+\cdots+n_d\lambda_d<\beta\}.
\end{array}
\end{equation}

 Let $\beta=\alpha c\in \ZZ_+$, where $c$ is the constant of (\ref{eq8}), and $\alpha$ is the constant of Lemma \ref{Lemma7}, so that for all $i\in \ZZ_+$,
\begin{equation}\label{eq13}
K_{\beta i}\cap R= K_{\alpha c i}\cap R\subset m_R^{ic}\subset I_i.
\end{equation}
We have from (\ref{eq12}) that
\begin{equation}\label{eq14}
\begin{array}{lll}
\ell_R(R/I_i)&=& \#\{(n_1,\ldots,n_d)\in\phi(R)\mid n_1\lambda_1+\cdots+n_d\lambda_d<\beta i\}\\
&&-\#\{(n_1,\ldots,n_d)\in\phi(I_i)\mid n_1\lambda_1+\cdots+n_d\lambda_d<\beta i\}.
\end{array}
\end{equation}

Now $(n_1,\ldots,n_d)\in\phi(R)$ and $n_1+\cdots+n_d\ge \beta i$ implies $n_1\lambda_1+\cdots+n_d\lambda_d\ge\beta i$ by (\ref{eq9}), so that $(n_1,\ldots,n_d)\in \phi(I_i)$ by (\ref{eq13}). Thus 
\begin{equation}\label{eq15}
\begin{array}{lll}
\ell_R(R/I_i)&=& \#\{(n_1,\ldots,n_d)\in\phi(R)\mid n_1+\cdots+n_d\le\beta i\}\\
&&-\#\{(n_1,\ldots,n_d)\in\phi(I_i)\mid n_1+\cdots+n_d\le\beta i\}.
\end{array}
\end{equation}
Let $\Gamma\subset \NN^{d+1}$ be the semigroup
$$
\Gamma=\{(n_1,\ldots,n_d,i)\mid (n_1,\ldots,n_d)\in \phi(I_i)\mbox{ and }n_1+\cdots+n_d\le \beta i\}.
$$
$I_0=R$ (and $\nu(1)=0$)  implies (\ref{Cone1}) holds. The semigroup
$$
B=\{(n_1,\ldots,n_d,i)\mid (n_1,\ldots,n_d)\in \NN^{d}\mbox{ and }n_1+\cdots+n_d\le\beta i\}
$$
is generated by $B\cap (\NN^d\times\{1\})$ and contains $\Gamma$, so (\ref{Cone2}) holds.

Write $y_i=\frac{f_i}{g_i}$ with $f_i,g_i\in R$ for $1\le i\le d$. Let $0\ne h\in I_1$. Then $hf_i, hg_i\in I_1$. There exists $c'\in \ZZ_+$ such that $c'\ge c$ and $hf_i, hg_i\not\in m_R^{c'}$ for $1\le i\le d$. We may replace $c$ with $c'$ in (\ref{eq8}). Then
$\phi(hf_i), \phi(hg_i)\in \Gamma_1=\Gamma\cap(\NN^d\times\{1\})$ for $1\le i\le d$, since $hf_i$ and $hg_i$ all have values $n_1\lambda_1+\cdots+n_d\lambda_d<\beta i$, so that $n_1+\ldots+n_d<\beta i$. We have that  $\phi(hf_i)-\phi(hg_i)=e_i$ for $1\le i\le d$. Thus
$$
(e_i,0)=(\phi(hf_i),1)-(\phi(hg_i),1)\in G(\Gamma)
$$
for $1\le i\le d$. Since $(\phi(hf_i),1)\in G(\Gamma)$, we have that $(0,1)\in G(\Gamma)$, so $G(\Gamma)=\ZZ^{d+1}$ and (\ref{Cone3}) holds.
By Theorem \ref{ConeTheorem1}, 
\begin{equation}\label{eq16}
\lim_{i\rightarrow\infty} \frac{\# \Gamma_i}{i^d}={\rm vol}(\Delta(\Gamma)).
\end{equation}

Let $\Gamma'\subset \NN^{d+1}$ be the semigroup
$$
\Gamma'=\{(n_1,\ldots,n_d,i)\mid (n_1,\ldots,n_d)\in \phi(R)\mbox{ and }n_1+\cdots+n_d\le \beta i\}.
$$
Our calculation for $\Gamma$ shows that (\ref{Cone1}) - (\ref{Cone3}) holds for $\Gamma'$. 
By Theorem \ref{ConeTheorem1}, 
\begin{equation}\label{eq17}
\lim_{i\rightarrow\infty} \frac{\# \Gamma'_i}{i^d}={\rm vol}(\Delta(\Gamma')).
\end{equation}
We obtain the conclusions of Theorem \ref{Theorem1} from equations (\ref{eq15}), (\ref{eq16}) and (\ref{eq17}).

The following is an immediate consequence of Theorem \ref{Theorem1}, taking $S=R$.

\begin{Theorem}\label{Theorem0} Suppose that $R$ is a  regular local ring of dimension $d>0$ and $\{I_n\}$ is a graded family of $m_R$-primary ideals in $R$. Then the limit
$$
\lim_{n\rightarrow\infty} \frac{\ell_R(R/I_n)}{n^d}
$$
exists.
\end{Theorem}

\section{A theorem on asymptotic lengths  in more general rings}

\begin{Lemma}\label{Lemma5} Suppose that $R$ is a $d$-dimensional reduced local ring and $\{I_n\}$ is a graded family of $m_R$-primary ideals in $R$,  Let $\{p_1,\ldots, p_s\}$ be the minimal primes of $R$, $R_i=R/p_i$, and let $S$ be the ring
$S=\bigoplus_{i=1}^sR_i$. Then there exists $\alpha\in \ZZ_+$ such that for all $n\in \ZZ_+$,
$$
|\sum_{i=1}^s\ell_{R_i}(R_i/I_nR_i)-\ell_R(R/I_n)|\le \alpha n^{d-1}.
$$
\end{Lemma}

\begin{proof}  
There exists $c\in \ZZ_+$ such that $m_R^c\subset I_1$. Since $S$ is a finitely generated $R$-submodule of the total ring of fractions $T=\bigoplus_{i=1}^s Q(R_i)$ of $R$, there exists a non zero divisor
$x\in R$ such that $xS\subset R$.

The natural inclusion $R\rightarrow S$ induces exact sequences of $R$-modules
\begin{equation}\label{eq18}
0\rightarrow R\cap I_nS/I_n\rightarrow R/I_n\rightarrow S/I_nS\rightarrow N_n\rightarrow 0.
\end{equation}

 We also have exact sequences of $R$-modules
\begin{equation}\label{eq19}
0\rightarrow A_n\rightarrow R/I_n\stackrel{x}{\rightarrow} R/I_n\rightarrow M_n\rightarrow 0.
\end{equation} 

 We have that $x(R\cap I_nS)\subset I_n$ and $A_n=I_n:x/I_n$, so that 
\begin{equation}\label{eq20}
\ell_R(R\cap I_nS/I_n)\le \ell_R(A_n).
\end{equation}

Now $M_n\cong (R/x)/I_n(R/x)$, so
$$
\ell_R(M_n)\le \ell_R((R/x)/m_R^{nc}(R/x))\le \beta (nc)^{d-1}
$$
for some $\beta$, computed from the Hilbert-Samuel polynomial of $R/x$ and the finitely many values of the Hilbert-Samuel function of $R/x$ which do not
agree with this polynomial. Thus
\begin{equation}\label{eq21}
\ell_R(A_n)=\ell(M_n)\le \beta c^{d-1}n^{d-1}
\end{equation}
by (\ref{eq19}).

Since $xS\subset R$, we have that
$$
N_n\cong (S/R+I_nS)=S/(R+I_nS+xS).
$$
Thus
\begin{equation}\label{eq22}
\ell_R(N_n)\le \ell_R((S/xS)/m_R^{nc}(S/xS))\le\gamma(nc)^{d-1}
\end{equation}
for some $\gamma$, computed from the Hilbert-Samuel polynomial of the semilocal ring $S/x$, with respect to the  ideal of definition $m_R(S/xS)$.
Thus
$$
|\ell_R(R/I_n)-\ell_R(S/I_nS)|\le \max\{\beta,\gamma\}c^{d-1}n^{d-1}.
$$
The lemma now follows, since
$$
\ell_R(S/I_nS)=\sum\ell_{R_i}(R_i/I_nR_i).
$$
\end{proof}

\begin{Theorem}\label{Theorem6} Suppose that $R$ is an analytically unramified  local ring with algebraically closed residue field.
Let $d>0$ be the dimension of $R$. Suppose that  $\{I_n\}$ is a graded family of $m_R$-primary ideals in $R$.
 Then
$$
\lim_{i\rightarrow\infty}\frac{\ell_R(R/I_i)}{i^d}
$$
exists.
\end{Theorem}

\begin{proof}  Let $\hat R$ be the $m_R$-adic completion of $R$, which is reduced and excellent. Since the $I_n$ are $m_R$-primary, we have that 
$R/I_n\cong \hat R/I_n\hat R$ and $\ell_R(R/I_n)=\ell_{\hat R}(\hat R/I_n\hat R)$ for all $n$.  Let
$\{q_1,\ldots,q_s\}$ be the minimal primes of $\hat R$. By Lemma \ref{Lemma5}, we reduce to proving the theorem for the families of ideals $\{I_n\hat R/q_i\}$ in
$\hat R/q_i$, for $1\le i\le s$. We may thus assume that $R$ is a complete domain.
Let $\pi:X\rightarrow \mbox{spec}(R)$ be the normalization of the blow up of $m_R$. $X$ is of finite type over $R$ since $R$ is excellent.
Since $\pi^{-1}(m_R)$ has codimension 1 in $X$ and $X$ is normal, there exists a closed point $x\in X$ such that the local ring $\mathcal O_{X,x}$ is a regular local ring. Let $S$ be this local ring. $S/m_S=R/m_R$ since $S/m_S$ is finite over $R/m_R$ which is an algebraically closed field.
\end{proof}

\begin{Lemma}\label{Lemma1} Suppose that $R$ is a Noetherian local domain that contains a field $k$. Suppose that $k'$ is a finite separable field extension of $k$ such that $k\subset R/m_R\subset k'$. Let $S=R\otimes_kk'$. Then $S$ is a reduced Noetherian semi local ring. Let $p_1,\ldots, p_r$ be the maximal ideals of $S$. Then $m_RS= p_1\cap \cdots \cap p_r$.
\end{Lemma}

\begin{proof} Let $K$ be the quotient field of $R$. Then $K\otimes_kk'$ is reduced (by Theorem 39, page 195 \cite{ZS1}). Since $k'$ is flat over $k$, we have an inclusion $R\otimes_kk'\subset K\otimes_kk'$, so $S=R\otimes_kk'$ is reduced. $S/m_RS\cong (R/m_R)\otimes_kk'$ is also reduced by Theorem 39, \cite{ZS1}. Thus
$m_RS=p_1\cap\cdots \cap p_r$.
\end{proof}

\begin{Remark}\label{Remark3} In the case that $R$ is a regular local ring, we have that $S=R\otimes_kk'$ is a regular ring.
\end{Remark}

\begin{proof} Since $R$ is a regular local ring, $m_R$ is generated by $d=\dim R$ elements. For $1\le i\le r$, we thus have that $p_iS_{p_i}=m_RS_{p_i}$ is generated by $d=\dim R=\dim S_{p_i}$ elements. Thus $S_{p_i}$ is a regular local ring.
\end{proof}

\begin{Remark}\label{Remark4} If $k'$ is Galois over $k$, then $S/p_i\cong k'$ for $1\le i\le r$.
\end{Remark}

\begin{proof} Let $\tilde k=R/m_R$. By our assumption, $\tilde k$ is a finite separable extension of $k$. Thus $\tilde k=k[\alpha]$ for some $\alpha\in k'$.
Let $f(x)\in k[x]$ be the minimal polynomial of $\alpha$. $k'$ is a normal extension of $k$ containing $\alpha$, so $f(x)$ splits into linear factors in $k'[x]$.
Thus
$$
\bigoplus_{i=1}^r R/p_i\cong S/m_RS\cong \tilde k\otimes_kk'\cong k'[x]/(f(x))\cong (k')^r.
$$
\end{proof}

\begin{Remark}\label{Remark5} If $R$ is complete in the $m_R$-adic topology, then $R\otimes_kk'$ is complete in the $m_RR\otimes_kk'$-adic topology
(Theorem 16, page 277 \cite{ZS2}). If $p_1,\ldots,p_r$ are the maximal ideals of $R\otimes_kk'$, then $R\otimes_kk'\cong \bigoplus_{i=1}^r(R\otimes_kk')_{p_i}$
(Theorem 8.15 \cite{Ma2}). Thus each $(R\otimes_kk')_{p_i}$ is a complete local ring.
\end{Remark}

\begin{Lemma}\label{Lemma2} Let assumptions and notation be as in Lemma \ref{Lemma1}, and suppose that $I$ is an $m_R$-primary ideal in $R$.
 Then
$$
[k':k]\ell_R(R/I)=\sum_{i=1}^r[S/p_i:R/m_R]\ell_{S_{p_i}}((S/IS)_{p_i}).
$$
\end{Lemma}

\begin{proof}
$$
\dim_kR/I=[R/m_R:k]\ell_R(R/I)
$$
and 
$$
\dim_kS/IS=\dim_k (R/I)\otimes_kk'=[k':k]\dim_k(R/I).
$$
$S/IS$ is an Artin local ring, so that $S/IS\cong \bigoplus_{i=1}^r(S/IS)_{p_i}$. Thus
$$
\dim_k(S/IS)=\sum_{i=1}^r [S/p_i:k]\ell_{S_{p_i}}((S/IS)_{p_i}).
$$
\end{proof}

We will need the following definition. A commutative ring $A$ containing a field $k$ is said to be geometrically irreducible over $k$ if 
$A\otimes_kk'$ has a unique minimal prime for all finite extensions $k'$ of $k$.

\begin{Theorem}\label{Theorem2}  Suppose that $R$ is an analytically unramified  equicharacteristic local ring with perfect residue field.
 Let $d>0$ be the dimension of $R$.  Suppose that
 $\{I_n\}$ is a graded family of $m_R$-primary ideals in $R$. Then
$$
\lim_{i\rightarrow\infty}\frac{\ell_R(R/I_i)}{i^d}
$$
exists.
\end{Theorem}

\begin{proof} 
There exists $c\in \ZZ_+$ such that $m_R^c\subset I_1$. Let $\hat R$ be the $m_R$-adic completion of $R$. Since the $I_n$ are $m_R$-primary, we have that 
$R/I_n\cong \hat R/I_n\hat R$ and $\ell_R(R/I_n)=\ell_{\hat R}(\hat R/I_n\hat R)$ for all $n$. $\hat R$ is reduced since $R$ is analytically unramified. Let
$\{q_1,\ldots,q_s\}$ be the minimal primes of $\hat R$. By Lemma \ref{Lemma5}, we reduce to proving the theorem for the families of ideals $\{I_n\hat R/q_i\}$ in
$\hat R/q_i$, for $1\le i\le s$. 
In the case of a minimal prime $q_i$ of $R$ such that $\dim R/q_i<d$, the limits
$$
\lim_{n\rightarrow \infty}\frac{\ell_R(R_i/I_nR_i)}{n^d}
$$
 are all zero, since
$\ell_R(R_i/I_nR_i)\le \ell_R(R_i/m_R^{nc}R_i)$ for all $n$.

We may thus assume that $R$ is a complete domain. $\hat R$ contains a coefficient field $k\cong R/m_R$ by the Cohen structure theorem,
as $R$ is complete and equicharacteristic. Let $k'$ be the separable closure of $k$ in $Q(R)$, and let $\overline R$ be the integral closure of $R$ in $Q(R)$.
We have that $k' \subset \overline R$.  $\overline R$ is a finitely generated $R$-module since $R$ is excellent. Let $n\subset \overline R$ be a maximal ideal lying over $m_R$. Then the residue field extension $R/m_R\rightarrow \overline R/n$ is finite. Since $k'\subset \overline R/n$, we have that $k'$ is a finite extension of $k$. By Corollary 4.5.11 \cite{EGAIV}, there exists a finite extension $L$ of $k$ (which can be taken to be Galois over $k$) such that if $q_1,\ldots,q_r$ are the minimal primes of $R\otimes_kL$, then each ring $R\otimes_kL/q_i$ is geometrically irreducible over $L$.

$R\otimes_kL$ is a reduced semi local ring by Lemma \ref{Lemma1}, and by Remark \ref{Remark4}, the residue field of all maximal ideals of $R\otimes_kL$ is $L$, which is a perfect field. By Remark \ref{Remark5} and Lemmas \ref{Lemma5} and \ref{Lemma2}, we reduce to the case where $R$ is a complete local domain
with perfect coefficient field $k$, such that $R$ is geometrically irreducible over $k$. Let $\pi:X\rightarrow \mbox{Spec}(R)$ be the normalization of the blow up of $m_R$. $\pi$ is projective and birational since $R$ is excellent. $m_R\mathcal O_X$ is locally principal, so $\pi^{-1}(m_R)$ has codimension 1 in $X$.
Since $X$ is normal, it is regular in codimension 1, so there exists a closed point $q\in X$ such that $\pi(q)=m_R$ and $S=\mathcal O_{X,q}$ is a regular local ring. Let $k'=S/m_S$. $k'$  is finite over $k$, and is thus a separable extension of the perfect field $k$.

Let $k''$ be a finite Galois extension of $k$ containing $k'$. Let $R'=R\otimes_kk''$. $R'$ is a local domain with residue field $k''$. $R'$ is complete by
Remark \ref{Remark5}. $S\otimes_kk''$ is regular and semi local by Remark \ref{Remark3}. Let $p\in S\otimes_kk''$ be a maximal ideal. Let $S'=(S\otimes_kk'')_p$.
There exist $f_0,\ldots,f_t\in Q(R)$ such that $S$ is a localization of $R[\frac{f_1}{f_0},\ldots,\frac{f_t}{f_0}]$ at a maximal ideal which necessarily contracts to $m_R$. Thus $S'$ essentially of finite type and birational over $R'$, since we can regard $f_0,\ldots,f_t\in R'$. Since $S'$ is a regular local ring and $k''=S'/m_{S'}=R'/m_{R'}$ by Remark \ref{Remark4}, we have that Theorem \ref{Theorem2} follows from Lemma \ref{Lemma2} and Theorem \ref{Theorem1}.
\end{proof}

\section{Some applications to asymptotic multiplicities}

\begin{Theorem}\label{Theorem4} Suppose that $R$ is a  local ring of dimension $d>0$ such that one of the following holds:
\begin{enumerate}
\item[1)] $R$ is regular or 
\item[2)] $R$ is analytically irreducible   with algebraically closed residue field or
\item[3)] $R$ is normal, excellent and equicharacteristic with perfect residue field.
\end{enumerate}
Suppose that $\{I_i\}$ and $\{J_i\}$ are graded families of nonzero ideals in $R$. Further suppose that $I_i\subset J_i$ for all $i$ and there exists $c\in\ZZ_+$ such that
\begin{equation}\label{eq60}
m_R^{ci}\cap I_i= m_R^{ci}\cap J_i
\end{equation}
 for all $i$. Then the limit
$$
\lim_{i\rightarrow \infty} \frac{\ell_R(J_i/I_i)}{i^{d}}
$$
exists.
\end{Theorem}

\begin{Remark} An analytic local domain $R$ satisfies the hypotheses of 2) of Theorem \ref{Theorem4}. 
The fact that $R$ is analytically irreducible ($\hat R$ is a domain) follows from Corollary 18.9.2 \cite{EGAIV}.
\end{Remark}

\begin{proof} We will apply the method of Theorem \ref{Theorem1}. When $R$ is regular we take $S=R$ and in Case 2), we construct $S$ by the argument of the proof of Theorem \ref{Theorem6}. We will consider Case 3) at the end of the proof.

Let $\nu$ be the valuation of $Q(R)$ constructed from $S$ in the proof of Theorem \ref{Theorem1}, with associated valuation ideals $K_{\lambda}$ in the valuation ring $V_{\nu}$ of $\nu$. 

Apply Lemma \ref{Lemma7} if $R$ is not regular, to find
$\alpha\in \ZZ_+$ such that 
$$
K_{\alpha n}\cap R\subset m_R^n
$$
for all $n\in \ZZ_+$. When $R$ is regular, so that $R=S$, the existence of such an $\alpha$ follows directly from (\ref{eq10}).
We will use the function $\phi:R\setminus\{0\}\rightarrow \NN^{d+1}$ of the proof of Theorem \ref{Theorem1}.
We have that 
$$
K_{\alpha cn}\cap I_n=K_{\alpha cn}\cap J_n
$$
for all $n$. Thus
\begin{equation}\label{eq31}
\ell_R(J_n/I_n)=\ell_R(J_n/K_{\alpha cn}\cap J_n)-\ell_R(I_n/K_{\alpha cn}\cap I_n)
\end{equation}
for all $n$. Let $\beta=\alpha c$ and 
$$
\Gamma(J_*)=\{(n_1,\ldots,n_d,i)\mid (n_1,\ldots,n_d)\in \phi (J_i)\mbox{ and }n_1+\cdots+n_d\le \beta i\}, 
$$
and
$$
\Gamma(I_*)=\{(n_1,\ldots,n_d,i)\mid (n_1,\ldots,n_d)\in \phi (I_i)\mbox{ and }n_1+\cdots+n_d\le \beta i\}.
$$
We have that 
\begin{equation}\label{eq32}
\ell_R(J_n/I_n)=\# \Gamma(J_*)_n-\# \Gamma(I_*)_n
\end{equation}
as explained in the proof of Theorem \ref{Theorem1}. As in the proof of Theorem \ref{Theorem1}, we have that $\Gamma(J_*)$ and $\Gamma(I_*)$ satisfy the conditions
(\ref{Cone1}) - (\ref{Cone3}). Thus
$$
\lim_{n\rightarrow \infty} \frac{\# \Gamma(J_*)_n}{n^d}={\rm vol}(\Delta(\Gamma(J_*))\mbox{ and }\lim_{n\rightarrow \infty} \frac{\# \Gamma(I_*)_n}{n^d}
={\rm vol}(\Delta(\Gamma(I_*))
$$
by Theorem \ref{ConeTheorem1}. 
The theorem (in Cases 1) or 2))now follows from (\ref{eq32}).

Now suppose that $R$ satisfies the assumptions of Case 3). Then the $m_R$-adic  completion $\hat R$ 
satisfies the assumptions of Case 3). 

Suppose that $R$ satisfies the assumptions of Case 3), and $R$ is $m_R$-adically complete. Let $k$ be a coefficient field of $R$. 
The algebraic closure of $k$ in $Q(R)$ is contained in $R$, so it is contained in $R/m_R=k$. Thus $k$ is algebraically closed in $Q(R)$.
Suppose that $k'$ is a finite Galois extension of $k$.  $Q(R)\otimes_kk'$ is a field by Corollary 2, page 198 \cite{ZS1}, and thus $R'=R\otimes_kk'$ is a domain.
$R'$ is a local ring with residue field $k'$ since $R'/m_RR'\cong R/m_R\otimes_kk'\cong k'$.  $R'$ is normal by Corollary 6.14.2 \cite{EGAIV}. Thus $R'$ satisfies the assumptions of Case 3).

Thus in the reductions in the proof of Theorem \ref{Theorem2} to \ref{Theorem1}, the only extensions which we need consider are local homomorphisms
$R\rightarrow R'$ which are either $m_R$-adic completion or a base extension by a Galois field extension. These extensions are all flat, and
$m_RR'=m_{R'}$. Thus
$$
m_S^{nc}\cap I_nS=m_R^{nc}S\cap I_nS=(m_R^{nc}\cap I_n)S=(m_R^{nc}\cap J_n)S=m_R^{nc}S\cap J_nS=m_S^{nc}\cap J_nS
$$
for all $n$. Thus the condition (\ref{eq60}) is preserved, so we reduce to the Case 2) of this theorem, and conclude that Theorem \ref{Theorem4} is true in Case 3).

\end{proof}

If $R$ is a local ring and $I$ is an ideal in $R$ then the saturation of $I$ is 
$$
I^{\rm sat}=I:m_R^{\infty}=\cup_{k=1}^{\infty}I:m_R^k.
$$

\begin{Corollary}\label{Corollary5} Suppose that $R$ is a  local ring of dimension $d>0$ such that  one of the following holds: 
\begin{enumerate}
\item[1)] $R$ is regular or 
\item[2)] $R$ is analytically irreducible  with algebraically closed residue field or
\item[3)] $R$ is normal, excellent and equicharacteristic with perfect residue field.
\end{enumerate}
Suppose that $I$ is an ideal in $R$. Then the limit
$$
\lim_{i\rightarrow \infty} \frac{\ell_R((I^i)^{\rm sat}/I^i)}{i^{d}}
$$
exists.

\end{Corollary}

Since $(I^n)^{\rm sat}/I^n\cong H^0_{m_R}(R/I^n)$, the epsilon multiplicity of Ulrich and Validashti \cite{UV}
$$
\epsilon(I)=\limsup \frac{\ell_R(H^0_{m_R}(R/I^n))}{n^d/d!}
$$
exists as a limit, under the assumptions of Corollary \ref{Corollary5}.

Corollary \ref{Corollary5} is proven for more general families of modules when $R$ is a local domain which is essentially of finite type over a perfect field $k$ such that $R/m_R$ is algebraic over $k$ in \cite{C}. The limit in Corollary \ref{Corollary5} can be irrational, as shown in \cite{CHST}.

\begin{proof} By Theorem 3.4 \cite{S}, there exists $c\in\ZZ_+$ such that each power $I^n$ of $I$ has an irredundant primary decomposition 
$$
I^n=q_1(n)\cap\cdots\cap q_s(n)
$$
where $q_1(n)$ is $m_R$-primary and $m_R^{nc}\subset q_1(n)$ for all $n$. Since $(I^n)^{\rm sat}=q_2(n)\cap \cdots\cap q_s(n)$,
we have that 
$$
I^n\cap m_R^{nc}=m_R^{nc}\cap q_2(n)\cap \cdots \cap q_s(n)=m_R^{nc}\cap (I^n)^{\rm sat}
$$
for all $n\in \ZZ_+$. Thus the corollary follows from Theorem \ref{Theorem4}, taking $I_i=I^i$ and $J_i=(I^i)^{\rm sat}$.

\end{proof}

A stronger version of the previous corollary is true. 
The following corollary proves a formula proposed by Herzog, Puthenpurakal and Verma in the introduction to \cite{HPV}.

Suppose that $R$ is a ring, and $I,J$ are ideals in $R$. Then the $n^{\rm th}$ symbolic power of $I$ with respect to $J$ is
$$
I_n(J)=I^n:J^{\infty}=\cup_{i=1}^{\infty}I^n:J^i.
$$

\begin{Corollary}\label{Cor5} Suppose that $R$ is a local domain of dimension $d$ such that one of the following holds:
\begin{enumerate}
\item[1)] $R$ is regular or 
\item[2)] $R$ is normal and excellent of equicharacteristic 0 or
\item[3)] $R$ is essentially of finite type over a field  of characteristic zero.
\end{enumerate}
 Suppose that $I$ and $J$ are ideals in $R$.  
 Let $s$ be the constant  limit dimension of $I_n(J)/I^n$ for $n\gg 0$. Then
$$
\lim_{n\rightarrow \infty} \frac{e_{m_R}(I_n(J)/I^n)}{n^{d-s}}
$$
exists.
\end{Corollary}

\begin{proof} There exists a positive integer $n_0$ such that the  set of associated primes of $R/I^n$ stabilizes for
$n\ge n_0$ by  \cite{Br}. Let $\{p_1,\ldots, p_t\}$ be this set of associated primes.  We thus  have irredundant primary decompositions
for $n\ge n_0$,
\begin{equation}\label{eq**}
I^n=q_1(n)\cap \cdots \cap q_t(n),
\end{equation}
where $q_i(n)$ are $p_i$-primary.

We further have that 
\begin{equation}\label{eq*}
I^n:J^{\infty}=\cap_{J\not\subset p_i}q_i(n).
\end{equation}
Thus $\dim I_n(J)/I^n$ is constant for $n\ge n_0$. Let $s$ be this limit dimension. The set 
$$
A=\{p\in \cup_{n\ge n_0}{\rm Ass}(I_n(J)/I^n)\mid n\ge n_0\mbox{ and }\dim R/p=s\}
$$
is a finite set. Moreover, every such prime is in ${\rm Ass}(I_n(J)/I^n$ for all $n\ge n_0$. For $n\ge n_0$, we have 
by the additivity formula (V-2 \cite{Se} or Corollary 4.6.8, page 189 \cite{BH}), that
$$
e_{m_R}(I_n(J)/I^n)=\sum_{p}\ell_{R_p}((I_n(J)/I^n)_p)e(m_{R/p})
$$
where the sum is over the finite set of primes $p\in \mbox{Spec}(R)$ such that $\dim R/p=s$. This sum is thus over the finite set $A$.

Suppose that $p\in A$ and $n\ge n_0$. Then 
$$
I^n_p=\cap q_i(n)_p
$$
where the intersection is over the $q_i(n)$ such that $p_i\subset p$, and
$$
I_n(J)=\cap q_i(n)_p
$$
where the intersection is over the $q_i(n)$ such that $J\not\subset p_i$ and $p_i\subset p$. Thus  there exists an index $i_0$ such that $p_{i_0}=p$ and 
$$
I^n_p=q_{i_0}(n)_p\cap I_n(J)_p.
$$
By (\ref{eq**}), 
$$
(I^n_p)^{\rm sat}=I_n(J)_p
$$
for $n\ge n_0$. Since $R_p$ satisfies one of the conditions 1) or 3) of Theorem \ref{Theorem4}, or the conditions of Corollary 1.5 \cite{C}, and $\dim R_p=d-s$ (as $R$ is universally catenary), the limit
$$
\lim_{n\rightarrow \infty} \frac{\ell_R((I_n(J)/I_n)_p)}{n^{d-s}}
$$
exists.
\end{proof}

\begin{Theorem}\label{Theorem5} Suppose that $R$ is a $d$-dimensional local ring such that either
\begin{enumerate}
\item[1)] $R$ is regular or
\item[2)] $R$ is analytically unramified and  equicharacteristic, with perfect residue field.
\end{enumerate}
 Suppose that $\{I_i\}$ is a graded family of $m_R$-primary ideals in $R$. Then 
$$
\lim_{n\rightarrow \infty}\frac{\ell_R(R/I_n)}{n^d/d!}=\lim_{p\rightarrow \infty}\frac{e(I_p)}{p^d}.
$$
Here $e(I_p)$ is the multiplicity
$$
e(I_p)=e_{I_p}(R)=\lim_{k\rightarrow \infty} \frac{\ell_R(R/I_p^k)}{k^d/d!}.
$$
\end{Theorem}

Theorem \ref{Theorem5} is proven for valuation ideals associated to an Abhyankar valuation in a regular local ring which is essentially of finite type over a field in  \cite{ELS}, for general families of $m_R$-primary ideals when $R$ is a regular local ring containing a field in \cite{Mus} and when $R$ is a local domain which is essentially of finite type over an algebraically closed field $k$ with $R/m_R=k$ in Theorem 3.8 \cite{LM}. 

\begin{proof} 
There exists $c\in\ZZ_+$ such that
$m_R^{c}\subset I_1$. 

We first prove the theorem when $R$ satisfies the assumptions of Theorem \ref{Theorem1}.
Let $\nu$ be the valuation of $Q(R)$ constructed from $S$ in the proof of Theorem \ref{Theorem1}, with associated valuation ideals $K_{\lambda}$ in the valuation ring $V_{\nu}$ of $\nu$. 

Apply Lemma \ref{Lemma7} if $R$ is not regular, to find
$\alpha\in \ZZ_+$ such that 
$$
K_{\alpha n}\cap R\subset m_R^n
$$
for all $n\in \NN$. When $R$ is regular, so that $R=S$, the existence of such an $\alpha$ follows directly from (\ref{eq10}).
We will use the function $\phi:R\setminus\{0\}\rightarrow \NN^{d+1}$ of the proof of Theorem \ref{Theorem1}.

We have that 
$$
K_{\alpha c n}\cap R \subset m_R^{cn}\subset I_n
$$
 for all $n$.

Let
$$
\Gamma(I_*)=\{(n_1,\ldots,n_d,i)\mid (n_1,\ldots,n_d)\in \phi(I_i)\mbox{ and }n_1+\cdots+n_d\le\alpha c i\},
$$
and
$$
\Gamma(R)=\{(n_1,\ldots,n_d,i)\mid (n_1,\ldots,n_d)\in \phi(R)\mbox{ and }n_1+\cdots+n_d\le\alpha c i\}.
$$
As in the proof of Theorem \ref{Theorem1}, $\Gamma(I_*)$ and $\Gamma(R)$ satisfy the conditions (\ref{Cone1}) - (\ref{Cone3}).
For fixed $p\in \ZZ_+$, let
$$
\Gamma(I_*)(p)=\{(n_1,\ldots,n_d,kp)\mid (n_1,\ldots,n_d)\in \phi(I_p^k)\mbox{ and }n_1+\cdots+n_d\le\alpha c kp\}.
$$
We have inclusions of semigroups
$$
k\Gamma(I_*)_p\subset \Gamma(I_*)(p)_{kp}\subset \Gamma(I_*)_{kp}
$$
for all $p$ and $k$.

By Theorem \ref{ConeTheorem2}, given $\epsilon>0$, there exists $p_0$ such that $p\ge p_0$ implies
$$
{\rm vol}(\Delta(\Gamma(I_*))-\epsilon\le \lim_{k\rightarrow \infty}\frac{\#k\Gamma(I_*)_p}{k^dp^d}.
$$
Thus
$$
{\rm vol}(\Delta(\Gamma(I_*))-\epsilon\le \lim_{k\rightarrow\infty}\frac{\#\Gamma(I_*)(p)_{kp}}{k^dp^d}
\le{\rm vol}(\Delta(\Gamma(I_*)).
$$
Again by Theorem \ref{ConeTheorem2}, we can choose $p_0$ sufficiently large that we also have that
$$
{\rm vol}(\Delta(\Gamma(R))-\epsilon\le\lim_{k\rightarrow \infty}\frac{\#\Gamma(R)_{kp}}{k^dp^d}\le{\rm vol}(\Delta(\Gamma)).
$$
Now 
$$
\ell_R(R/I_p^k)=\#\Gamma(R)_{pk}-\#\Gamma(I_*)(p)_{kp}
$$
and 
$$
\ell_R(R/I_n)=\#\Gamma(R)_n-\#\Gamma(I_*)_n.
$$
By Theorem \ref{ConeTheorem1},
$$
\lim_{n\rightarrow \infty}\frac{\ell_R(R/I_n)}{n^d}={\rm vol}(\Delta(\Gamma(R))-{\rm vol}(\Delta(\Gamma(I_*))).
$$
Thus
$$
\lim_{n\rightarrow\infty}\frac{\ell_R(R/I_n)}{n^d}-\epsilon\le \lim_{k\rightarrow\infty}\frac{\ell_R(R/I_p^k)}{k^dp^d}
=\frac{e(I_p)}{d!p^d}\le \lim_{n\rightarrow \infty}\frac{\ell_R(R/I_n)}{n^d}+\epsilon.
$$
Taking the limit as $p\rightarrow \infty$, we obtain the conclusions of the theorem.

Now assume that $R$ is general, satisfying the assumptions of the theorem. 
We reduce to the above case by a series of reductions, first taking the completion of $R$, then moding out by minimal primes, and by taking a base extension by
a finite Galois extension.

The proof thus reduces to showing that 
$$
\lim_{p\rightarrow \infty} \frac{e_d(I_p,R)}{p^d}=\lim_{n\rightarrow \infty}\frac{\ell_R(R/I_n)}{n^d/d!}
$$
in each of the following cases:

\begin{enumerate}
\item[a)] 
$$
\lim_{p\rightarrow \infty} \frac{e_d(I_p\hat R,\hat R)}{p^d}=\lim_{n\rightarrow \infty}\frac{\ell_{\hat R}(\hat R/I_n\hat R)}{n^d/d!}.
$$
\item[b)] Suppose that the minimal primes of (the reduced ring) $R$ are $\{q_1,\ldots, q_s\}$. Let  $R_i=R/q_i$ and suppose that 
$$
\lim_{p\rightarrow \infty} \frac{e_d(I_pR_i,R_i)}{p^d}=\lim_{n\rightarrow \infty}\frac{\ell_{R_i}(R_i/I_nR_i)}{n^d/d!}
$$
for all $i$.
\item[c)] Suppose that $k\subset R$ is a field and $k'$ is a finite Galois extension of $k$ containing $R/m_R$. Let $\{p_1,\ldots,p_r\}$ be the maximal ideals of the semi-local ring $S=R\otimes_kk'$. Suppose that 
$$
\lim_{p\rightarrow \infty} \frac{e_d(I_pS_{p_i},S_{p_i})}{p^d}=\lim_{n\rightarrow \infty}\frac{\ell_{S_{p_i}}(S_{p_i}/I_nS_{p_i})}{n^d/d!}.
$$
for all $i$.
\end{enumerate}

Recall that 
$$
\frac{e_d(I_p,R)}{d!}=\lim_{k\rightarrow \infty}\frac{\ell_R(R/I_p^k)}{k^d}.
$$

Case a) follows since 
$$
\ell_R(R/I_p^k)=\ell_{\hat R}(\hat R/I_p^k\hat R)
$$
for all $p,k$.

In Case b), we have that
$$
\frac{e_d(I_p,R)}{p^d}=\sum_{i=1}^s\frac{e_d(I_pR_i,R_i)}{p^d}
$$
by the additivity  formula (page V-3 \cite{Se} or Corollary 4.6.8, page 189 \cite{BH}) or directly from Lemma \ref{Lemma5}. Case b) thus follows from  the fact that
$$
\lim_{n\rightarrow\infty}\frac{\ell_R(R/I_n)}{n^d}=\sum_{i=1}^s\lim_{k\rightarrow\infty}\frac{\ell_R(R_i/I_nR_i)}{n^d}
$$
by Lemma \ref{Lemma5}.

In Case c) we have that $k'$ is Galois over $k$, so that $S/p_i\cong k'$ for all $i$ by Remark \ref{Remark4}. Thus Lemma \ref{Lemma2} becomes
$$
\ell_R(R/I_p^k)=\sum_{i=1}^r\ell_{S_{p_i}}(S_{p_i}/I_p^kS_{p_i})
$$
for all $p,k$, from which this case follows.

\end{proof}

Suppose that $R$ is a Noetherian ring, and $\{I_i\}$ is a graded  family of ideals in $R$.
Let
$$
s=s(I_*)=\limsup \dim R/ I_i.
$$
Let $i_0\in \ZZ_+$ be the smallest integer such that
\begin{equation}\label{eq40}
\mbox{$\dim R/I_i \le  s$ for $i\ge i_0$.}
\end{equation}
For $i\ge i_0$ and $p$ a prime ideal in $R$ such that $\dim R/p=s$, we have that $(I_i)_p=R_p$ or $(I_i)_p$ is $p_p$-primary.

$s$ is in general not a limit, as is shown by the following simple example.

\begin{Example} Suppose that $R$ is a Noetherian  ring and $p\subset q\subset R$ are prime ideals. Let
$$
I_i=\left\{\begin{array}{ll}
p&\mbox{ if $i$ is odd}\\
q&\mbox{ if $i$ is even}
\end{array}\right.
$$ 
We have that 
$$
I_iI_j=\left\{\begin{array}{ll}
p^2\mbox{ or }q^2&\mbox{ if $i+j$ is even}\\
pq&\mbox{ if $i+j$ is odd.}
\end{array}\right.
$$
Thus $I_iI_j\subset I_{i+j}$ for all $i,j$, and
$$
\dim R/I_i=\left\{\begin{array}{ll}
\dim R/p&\mbox{ if $i$ is odd}\\
\dim R/q&\mbox{ if $i$ is even}
\end{array}
\right.
$$
\end{Example}

Let
$$
T=T(I_*)=\{p\in \mbox{spec}(R)\mid \dim R/p=s\mbox{ and there exist arbitrarily large $j$ such that $(I_j)_p\ne R_p$}\}.
$$

\begin{Lemma}\label{Lemma10}  $T(I_*)$ is a finite set.
\end{Lemma}

\begin{proof} Let $U$ be the set of prime ideals $p$ of $R$ which are an associated prime of some $I_i$ with $i_0\le i\le 2i_0-1$, and ${\rm ht}\,  p=s$.
Suppose that $q\in T$. There exists $j\ge i_0$ such that $(I_j)_q\ne R_q$. We can write $j=ai_0+(i_0+k)$ with  $0\le k\le i_0-1$ and $a\ge 0$. Thus $I_{i_0}^aI_{i_0+k}\subset I_j$. Thus $q\in U$ since $(I_{i_0}^aI_{i_0+k})_q\ne R_q$.
\end{proof}

\begin{Lemma}\label{Lemma11} There exist $c=c(I_*)\in \ZZ_+$ such that if $j\ge i_0$ and $p\in T(I_*)$, then 
$$
p^{jc}R_p\subset I_jR_p.
$$
\end{Lemma}

\begin{proof} There exists $a\in \ZZ_+$ such that for all $p\in T$, $p_p^a\subset (I_i)_p$ for $i_0\le i\le 2i_0-1$.

Write $j=ti_0+(i_0+k)$ with $t\ge 0$ $0\le k\le i_0-1$. Then
$$
p_p^{(t+1)a}\subset I_{i_0}^tI_{i_0+k}R_p\subset I_jR_p.
$$
Let $c=\lceil \frac{a}{i_0}\rceil+a$. 
$$
jc\ge a+j\frac{a}{i_0}=a+(ti_0+i_0+k)\frac{a}{i_0}
\ge (t+1)a.
$$
Thus $p_p^{jc}\subset p_p^{(t+1)a}\subset (I_j)_p$.

\end{proof}

Let 
$$
A(I_*)=\{q\in T(I_*)\mid  \mbox{$I_nR_q$ is $q_q$-primary for  $n\ge i_0$}\}.
$$

\begin{Lemma}\label{Lemma50} Suppose that $q\in T(I_*)\setminus A(I_*)$. Then there exists $b\in \ZZ_+$ such that $q_q^b\subset (I_n)_q$ for all $n\ge i_0$.
\end{Lemma}

\begin{proof} There exists $n_0\in \ZZ_+$ such that $n_0\ge i_0$ and $(I_{n_0})_q=R_q$. Let $b\in \ZZ_+$ be such that $q_q^b\subset (I_n)_q$ for $0\le n<n_0$. 
Suppose that $i_0\le n\ge n_0$. Write $n=\beta n_0+\alpha$ with $\beta\ge 0$ and $0\le \alpha<n_0$. Then
$$
q_q^b\subset (I_{n_0}^{\beta}I_{\alpha})_q\subset (I_n)_q.
$$
\end{proof}

We obtain the following asymptotic additivity formula.

\begin{Theorem}\label{Theorem3} Suppose that $R$ is a $d$-dimensional local ring such that either 
\begin{enumerate}
\item[1)] $R$ is regular or 
\item[2)] $R$ is analytically unramified of equicharacteristic 0.
\end{enumerate}
Suppose that $\{I_i\}$ is a graded family of  ideals in $R$. Let
$s=s(I_*)=\limsup \dim R/I_i$ and $A=A(I_*)$. Suppose that $s<d$. Then 
$$
\lim_{n\rightarrow \infty}\frac{e_{s}(m_R,R/I_n)}{n^{d-s}/(d-s)!}=\sum_{q\in A}\left(\lim_{k\rightarrow \infty}\frac{e((I_k)_q)}{k^{d-s}}\right) e(m_{R/q}).
$$
\end{Theorem}

\begin{proof} 
 Let $i_0$ be the (smallest) constant satisfying (\ref{eq40}). By the additivity formula (V-2 \cite{Se} or Corollary 4.6.8, page 189 \cite{BH}), for $i\ge i_0$,
$$
e_{s}(m_R,R/I_i)=\sum_p\ell_{R_p}(R_p/(I_i)_p)e_{m_R}(R/p)
$$
where the sum is over all prime ideals $p$ of $R$ with $\dim R/p=s$. By Lemma \ref{Lemma10}, for $i\ge i_0$,
the sum is actually over the finite set $T(I_*)$ of prime ideals of $R$.

For $p \in T(I_*)$, $R_p$ is a local ring of dimension $\le d-s$. Further, $R_p$  is analytically unramified
(by \cite{R3} or Prop 9.1.4 \cite{SH}). By Lemma \ref{Lemma11}, and by Theorem \ref{Theorem0} or Theorem \ref{Theorem2},  replacing $(I_i)_p$ with $p_p^{ic}$ if $i<i_0$, we have that
$$
\lim_{i\rightarrow \infty}\frac{\ell_{R_p}(R_p/(I_i)_p)}{i^{d-s}}
$$
exists. Further, this limit is zero if $p\in T(I_*)\setminus A(I_*)$ by Lemma \ref{Lemma50}, and since $s<d$. Finally, we have
$$
\lim_{i\rightarrow \infty}\frac{\ell_{R_q}(R_q/(I_i)_q)}{i^{d-s}/(d-s)!}=\lim_{k\rightarrow \infty}\frac{e_{(I_k)_q}(R_q)}{k^{d-s}}
$$
for $q\in A(I_*)$ by Theorem \ref{Theorem5}.

\end{proof} 

\section{An application to growth of valuation semigroups}

As a consequence of our main result, we obtain the following theorem which gives a positive answer to a question raised in a recent paper by the author,
Kia Dalili and Olga Kashcheyeva \cite{CDK}. This formula was established if $R$ is a regular local ring of dimension 2 with algebraically closed residue field in \cite{CDK}, and if
$R$ is an arbitrary  regular local ring of dimension 2 in \cite{CV} using a detailed analysis of a generating sequence associated to the valuation.
A valuation $\omega$ dominating a local domain $R$ is zero dimensional if the residue field of $\omega$ is algebraic over $R/m_R$.

\begin{Theorem}\label{Theorem8} Suppose that $R$ is a  regular local ring or an analytically unramified  local domain. Further suppose that $R$ has an  algebraically closed residue field. Let $d>0$ be the dimension of $R$. Let $\omega$ be a zero dimensional rank one valuation of the quotient field of $R$ which dominates $R$. 
Let $S^R(\omega)$ be the semigroup of values of elements of $R$, which can be regarded as an ordered sub semigroup of $\RR_+$. For $n\in \ZZ_+$, define
$$
\phi(n)=|S^R(\omega)\cap (0,n)|.
$$
Then
$$
\lim_{n\rightarrow \infty}\frac{\phi(n)}{n^d}
$$
exists.
\end{Theorem}

\begin{proof} Let 
$$
I_n=\{f\in R\mid \omega(f)\ge n\}.
$$
Let 
$$
\lambda=\omega(m_R)=\min\{\omega(f)\mid f\in m_R\}.
$$
Let $c\in \ZZ_+$ be such that $c\lambda>1$. Then $m_R^{c}\subset I_1$.
By Theorem \ref{Theorem0} or  \ref{Theorem6}, we have that 
$$
\lim_{n\rightarrow \infty} \frac{\ell_R(R/I_n)}{n^d}
$$ 
exists.

Since $R$ has algebraically closed residue field, we have by  \cite{CDK} or \cite{CT}, that
$$
\#\phi(n)=\ell_R(R/I_n)-1.
$$
Thus the theorem follows.
\end{proof}

In \cite{CDK} it is shown that the real numbers $\beta$ with $0\le \beta<1/2$ are the limits attained on a regular local ring $R$ of dimension 2.

\end{document}